\documentclass[12pt]{amsart}

\usepackage{mathrsfs,amssymb}

\newtheorem{theorem}{Theorem}[section]
\newtheorem{lemma}[theorem]{Lemma}
\newtheorem{prop}[theorem]{Proposition}

\theoremstyle{definition}

\theoremstyle{remark}

\numberwithin{equation}{section}

\let \la=\lambda
\let \e=\varepsilon
\let \d=\delta
\let \o=\omega
\let \a=\alpha

\let \O=\Omega

\begin{document}
\title[$A_2$ conjecture]
{A simple proof of the $A_2$ conjecture}

\author{Andrei K. Lerner}
\address{Department of Mathematics,
Bar-Ilan University, 52900 Ramat Gan, Israel}
\email{aklerner@netvision.net.il}

\begin{abstract}
We give a simple proof of the $A_2$ conjecture proved recently by T. Hyt\"onen. Our proof avoids completely
the notion of the Haar shift operator, and it is based only on the ``local mean oscillation decomposition".
Also our proof yields a simple proof of the ``two-weight conjecture" as well.
\end{abstract}



\maketitle

\section{Introduction}

Let $T$ be an $L^2$ bounded Calder\'on-Zygmund operator. We say that $w\in A_2$ if
$$\|w\|_{A_2}=\sup_{Q\subset {\mathbb R}^n}w(Q)w^{-1}(Q)/|Q|^2<\infty.$$
In this note we give a rather simple proof of the $A_2$ conjecture recently settled
by T. Hyt\"onen \cite{H}.

\begin{theorem}\label{a2} For any $w\in A_2$,
\begin{equation}\label{lin}
\|T\|_{L^2(w)}\le c(n,T)\|w\|_{A_2}.
\end{equation}
\end{theorem}

Below is a partial list of important contributions to this result. First, (\ref{lin}) was proved for the
following operators:

\begin{list}{\labelitemi}{\leftmargin=1em}
\item Hardy-Littlewood maximal operator (S. Buckley \cite{B}, 1993);
\item Beurling transform (S. Petermichl and A. Volberg \cite{PV}, 2002);
\item Hilbert transform (S. Petermichl \cite{P1}, 2007);
\item Riesz transform (S. Petermichl \cite{P2}, 2008);
\item dyadic paraproduct (O. Beznosova \cite{Be}, 2008);
\item Haar shift (M. Lacey, S. Petermichl and M. Reguera \cite{LPR}, 2010).
\end{list}

After that, the following works appeared with very small intervals:

\begin{list}{\labelitemi}{\leftmargin=1em}
\item a simplified proof for Haar shifts (D. Cruz-Uribe, J. Martell and C.~P\'erez \cite{CMP1,CMP2}, 2010);
\item the $L^2(w)$ bound for general $T$ by $\|w\|_{A_2}\log(1+\|w\|_{A_2})$
(C. P\'erez, S. Treil and A. Volberg \cite{PTV}, 2010);
\item (\ref{lin}) in full generality (T. Hyt\"onen \cite{H}, 2010);
\item a simplification of the proof (T. Hyt\"onen et al. \cite{HPTV}, 2010);
\item (\ref{lin}) for the maximal Calder\'on-Zygmund operator $T_{\natural}$ (T. Hyt\"onen et al. \cite{HLMORSU}, 2010).
\end{list}

The ``Bellman function" proof of the $A_2$ conjecture in a geometrically doubling metric space
was given by F.~Nazarov, A. Reznikov and A.~Volberg \cite{NRV} (see also \cite{NV}). 

All currently known proofs of (\ref{a2}) were based on the
representation of $T$ in terms of the Haar shift operators ${\mathbb S}_{{\mathscr{D}}}^{m,k}$. Such representations also have a long history; for general $T$ it
was found in \cite{H}. The second key element of all known proofs was showing (\ref{a2}) for
${\mathbb S}_{{\mathscr{D}}}^{m,k}$ in place of $T$ with the corresponding constant depending linearly (or polynomially) on the complexity.
Observe that over the past year several different proofs of this step appeared (see, e.g., \cite{La,T}).

In a very recent work \cite{L2}, we have proved that for any Banach function space $X({\mathbb R}^n)$,
\begin{equation}\label{ineq}
\|T_{\natural}f\|_{X}\le c(T,n)\sup_{{\mathscr{D}},{\mathcal S}}\|{\mathcal A}_{{\mathscr{D}},{\mathcal S}}|f|\|_{X},
\end{equation}
where
$${\mathcal A}_{{\mathscr{D}},{\mathcal S}}f(x)=\sum_{j,k}f_{Q_j^k}\chi_{Q_j^k}(x)$$
(this operator is defined by means of a sparse family ${\mathcal S}=\{Q_j^k\}$ from a general dyadic grid ${\mathscr{D}}$; for these notions see
Section 2 below).

Observe that for the operator ${\mathcal A}_{{\mathscr{D}},{\mathcal S}}f$ inequality (\ref{lin}) follows just in few lines by a very simple
argument. This was first observed in \cite{CMP1,CMP2} (see also \cite{L2}).
Hence, in the case when $X=L^2(w)$, inequality (\ref{ineq}) easily implies the $A_2$ conjecture.
Also, (\ref{ineq}) yields the ``two-weight conjecture" by D. Cruz-Uribe and C. P\'erez; we refer to \cite{L2} for the details.

The proof of (\ref{ineq}) in \cite{L2} still depended on the
representation of $T$ in terms of the Haar shift operators. In this note we will show that this difficult step can be completely avoided.
Our new proof of (\ref{ineq}) is based only on the ``local mean oscillation decomposition" proved by the author in \cite{L1}. It is interesting
that we apply this decomposition twice. First it is applied directly to $T_{\natural}$, and we obtain that $T_{\natural}$ is essentially pointwise
dominated by the maximal operator $M$ and a series of dyadic type operators ${\mathcal T}_m$. In order to handle ${\mathcal T}_m$, we apply the decomposition again
to the adjoint operators ${\mathcal T}^{\star}_m$. After this step we obtain a pointwise domination by the simplest dyadic operators
${\mathcal A}_{{\mathscr{D}},{\mathcal S}}$.

Note that all our estimates are actually pointwise, and they do not depend on a particular function space. This explains why we prefer to write (\ref{ineq})
with a general Banach function space $X$.

\section{Preliminaries}
\subsection{Calder\'on-Zygmund operators}

By a Calder\'on-Zygmund operator in ${\mathbb R}^n$ we mean an $L^2$ bounded integral operator represented as
$$Tf(x)=\int_{{\mathbb R}^n}K(x,y)f(y)dy,\quad x\not\in\text{supp}\,f,$$
with kernel $K$ satisfying the following growth and smoothness conditions:
\begin{enumerate}
\renewcommand{\labelenumi}{(\roman{enumi})}
\item
$|K(x,y)|\le \frac{c}{|x-y|^n}$ for all $x\not=y$;
\item
there exists $0<\d\le 1$ such that
$$|K(x,y)-K(x',y)|+|K(y,x)-K(y,x')|\le
c\frac{|x-x'|^{\d}}{|x-y|^{n+\d}},$$ whenever $|x-x'|<|x-y|/2$.
\end{enumerate}

Given a Calder\'on-Zygmund operator $T$, define its maximal truncated version by
$$T_{\natural}f(x)=\sup_{0<\e<\nu}\Big|\int_{\e<|y|<\nu}K(x,y)f(y)dy\Big|.$$

\subsection{Dyadic grids}

Recall that the standard dyadic grid in ${\mathbb R}^n$ consists of the cubes
$$2^{-k}([0,1)^n+j),\quad k\in{\mathbb Z}, j\in{\mathbb Z}^n.$$
Denote the standard grid by ${\mathcal D}$.

By a {\it general dyadic grid} ${\mathscr{D}}$ we mean a collection of
cubes with the following properties: (i)
for any $Q\in {\mathscr{D}}$ its sidelength $\ell_Q$ is of the form
$2^k, k\in {\mathbb Z}$; (ii) $Q\cap R\in\{Q,R,\emptyset\}$ for any $Q,R\in {\mathscr{D}}$;
(iii) the cubes of a fixed sidelength $2^k$ form a partition of ${\mathbb
R}^n$.

Given a cube $Q_0$, denote by ${\mathcal D}(Q_0)$ the set of all
dyadic cubes with respect to $Q_0$, that is, the cubes from ${\mathcal D}(Q_0)$ are formed
by repeated subdivision of $Q_0$ and each of its descendants into $2^n$ congruent subcubes.
Observe that if $Q_0\in {\mathscr{D}}$, then each cube from ${\mathcal D}(Q_0)$ will also
belong to ${\mathscr{D}}$.

A well known principle says that there are $\xi_n$ general dyadic grids ${\mathscr{D}}_{\a}$ such that every cube $Q\subset {\mathbb R}^n$
is contained in some cube $Q'\in {\mathscr{D}}_{\a}$ such that $|Q'|\le c_n|Q|$. For $\xi_n=3^n$ this is attributed in the literature to M.~Christ and,
independently, to J. Garnett and P. Jones. For $\xi_n=2^n$ it can be found in a recent work by T. Hyt\"onen and C. P\'erez \cite{HP}. Very recently it was shown
by J. Conde et al. \cite{CGCP} that one can take $\xi_n=n+1$, and this number is optimal. For our purposes any of such variants is suitable. We will use the one from
\cite{HP}.

\begin{prop}\label{prhp} There are $2^n$ dyadic grids ${\mathscr{D}}_{\a}$ such that for any cube $Q\subset {\mathbb R}^n$ there exists a cube $Q_{\a}\in {\mathscr{D}}_{\a}$
such that $Q\subset Q_{\a}$ and $\ell_{Q_{\a}}\le 6\ell_Q$.
\end{prop}

The grids ${\mathscr{D}}_{\a}$ here are the following:
$${\mathscr{D}}_{\a}=\{2^{-k}([0,1)^n+j+\a)\},\quad\a\in\{0,1/3\}^n.$$

We outline briefly the proof. First, it is easy to see that it suffices to consider the one-dimensional case. Take an arbitrary interval
$I\subset {\mathbb R}$. Fix $k_0\in{\mathbb Z}$ such that $2^{-k_0-1}\le 3\ell_{I}<2^{-k_0}$. If $I$ does not contain any point $2^{-k_0}j,j\in {\mathbb Z}$,
then $I$ is contained in some $I'=[2^{-k_0}j,2^{-k_0}(j+1))$ (since such intervals form a partition of ${\mathbb R}$), and $\ell_{I'}\le 6\ell_{I}$.
On the other hand, if $I$ contains some point $j_02^{-k_0}$, then $I$ does not contain any point $2^{-k_0}(j+1/3),j\in {\mathbb Z}$ (since $\ell_I<2^{-k_0}/3$), and therefore $I$ is contained in some $I''=[2^{-k_0}(j+1/3),2^{-k_0}(j+4/3))$, and $\ell_{I''}\le 6\ell_{I}$.

\subsection{Local mean oscillations}
Given a measurable function $f$ on ${\mathbb R}^n$ and a cube $Q$,
the local mean oscillation of $f$ on $Q$ is defined by
$$\o_{\la}(f;Q)=\inf_{c\in {\mathbb R}}
\big((f-c)\chi_{Q}\big)^*\big(\la|Q|\big)\quad(0<\la<1),$$
where $f^*$ denotes the non-increasing rearrangement of $f$.

By a median value of $f$ over $Q$ we mean a possibly nonunique, real
number $m_f(Q)$ such that
$$\max\big(|\{x\in Q: f(x)>m_f(Q)\}|,|\{x\in Q: f(x)<m_f(Q)\}|\big)\le |Q|/2.$$

It is easy to see that the set of all median values of $f$ is either one point or the closed interval. In the latter case we will assume for
the definiteness that $m_f(Q)$ is the {\it maximal} median value. Observe that it follows from the definitions that
\begin{equation}\label{pro1}
|m_f(Q)|\le (f\chi_Q)^*(|Q|/2).
\end{equation}

Given a cube $Q_0$, the dyadic local sharp maximal
function $M^{\#,d}_{\la;Q_0}f$ is defined by
$$M^{\#,d}_{\la;Q_0}f(x)=\sup_{x\in Q'\in
{\mathcal D}(Q_0)}\o_{\la}(f;Q').$$

We say that $\{Q_j^k\}$ is a {\it sparse family} of cubes if:
(i)~the cubes $Q_j^k$ are disjoint in $j$, with $k$ fixed;
(ii) if $\Omega_k=\cup_jQ_j^k$, then $\Omega_{k+1}\subset~\Omega_k$;
(iii) $|\Omega_{k+1}\cap Q_j^k|\le \frac{1}{2}|Q_j^k|$.

The following theorem was proved in \cite{L2} (its very similar version can be found in \cite{L1}).

\begin{theorem}\label{decom1} Let $f$ be a measurable function on
${\mathbb R}^n$ and let $Q_0$ be a fixed cube. Then there exists a
(possibly empty) sparse family of cubes $Q_j^k\in {\mathcal D}(Q_0)$
such that for a.e. $x\in Q_0$,
$$
|f(x)-m_f(Q_0)|\le
4M_{\frac{1}{2^{n+2}};Q_0}^{\#,d}f(x)+2\sum_{k,j}
\o_{\frac{1}{2^{n+2}}}(f;Q_j^k)\chi_{Q_j^k}(x).
$$
\end{theorem}

The following proposition is well known, and it can be found in a slightly different form in \cite{JT}. We give its
proof here for the sake of the completeness. The proof is a classical argument used, for example, to show that $T$ is bounded
from $L^{\infty}$ to $BMO$. Also the same argument is used to prove a good-$\la$ inequality relating $T$ and $M$.

\begin{prop}\label{oscsin} For any cube $Q\subset {\mathbb R}^n$,
\begin{equation}\label{loc1}
\o_{\la}(Tf;Q)\le c(T,\la,n)\sum_{m=0}^{\infty}\frac{1}{2^{m\d}}\left(\frac{1}{|2^mQ|}\int_{2^mQ}|f(y)|dy\right)
\end{equation}
and
\begin{equation}\label{loc2}
\o_{\la}(T_{\natural}f;Q)\le c(T,\la,n)\sum_{m=0}^{\infty}\frac{1}{2^{m\d}}\left(\frac{1}{|2^mQ|}\int_{2^mQ}|f(y)|dy\right).
\end{equation}
\end{prop}

\begin{proof} Let $f_1=f\chi_{2\sqrt n Q}$ and $f_2=f-f_1$. If $x\in Q$ and $x_0$ is the center of $Q$, then
by the kernel assumptions,
\begin{eqnarray*}
&&|T(f_2)(x)-T(f_2)(x_0)|\le \int_{{\mathbb R}^n\setminus 2\sqrt n Q}|f(y)||K(x,y)-K(x_0,y)|dy\\
&&\le c\ell_Q^{\d}\int\limits_{{\mathbb R}^n\setminus 2Q}\frac{|f(y)|}{|x-y|^{n+\d}}dy\le c\ell_Q^{\d}
\sum_{m=0}^{\infty}\frac{1}{(2^m\ell_Q)^{n+\d}}\int\limits_{2^{m+1}Q\setminus 2^mQ}|f(y)|dy\\
&&\le c\sum_{m=0}^{\infty}\frac{1}{2^{m\d}}\left(\frac{1}{|2^mQ|}\int_{2^mQ}|f(y)|dy\right).
\end{eqnarray*}
From this and from the weak type $(1,1)$ of $T$,
\begin{eqnarray*}
&&\big(\big(Tf-T(f_2)(x_0)\big)\chi_Q\big)^*\big(\la|Q|\big)\\
&&\le (T(f_1))^*(\la|Q|)+\|T(f_2)-T(f_2)(x_0)\|_{L^{\infty}(Q)}\\
&&\le c\frac{1}{|Q|}\int_{2\sqrt n|Q|}|f(y)|dy+c\sum_{m=0}^{\infty}\frac{1}{2^{m\d}}\left(\frac{1}{|2^mQ|}\int_{2^mQ}|f(y)|dy\right)\\
&&\le c'\sum_{m=0}^{\infty}\frac{1}{2^{m\d}}\left(\frac{1}{|2^mQ|}\int_{2^mQ}|f(y)|dy\right),
\end{eqnarray*}
which proves (\ref{loc1}).

The same inequalities hold for $T_{\natural}$ as well, which gives (\ref{loc2}). The only trivial difference in the argument is that one needs to
use the sublinearity of $T_{\natural}$ instead of the linearity of $T$.
\end{proof}

\section{Proof of (\ref{ineq})}
Combining Proposition \ref{oscsin} and Theorem \ref{decom1} with $Q_0\in {\mathcal D}$, we get
that there exists a sparse family $S=\{Q_j^k\}\in {\mathcal D}$ such that for a.e. $x\in Q_0$,
$$|T_{\natural}f(x)-m_{Q_0}(T_{\natural}f)|\le c(n,T)\Big(Mf(x)+\sum_{m=0}^{\infty}\frac{1}{2^{m\d}}{\mathcal T}_{\mathcal S,m}|f|(x)\Big),$$
where $M$ is the Hardy-Littlewood maximal operator and
$${\mathcal T}_{\mathcal S,m}f(x)=\sum_{j,k}f_{2^mQ_j^k}\chi_{Q_j^k}(x).$$

If $f\in L^1$, then it follows from (\ref{pro1}) that $|m_Q(T_{\natural}f)|\to 0$ as $|Q|\to\infty$. Therefore, letting $Q_0$ to anyone of $2^n$ quadrants and
using Fatou's lemma, we get
$$
\|T_{\natural}f\|_{X}\le c(n,T)\Big(\|Mf\|_X+\sum_{m=0}^{\infty}\frac{1}{2^{m\d}}\sup_{{\mathcal S}\in {\mathcal D}}\|{\mathcal T}_{\mathcal S,m}|f|\|_{X}\Big)
$$
(for the notion of the Banach function space $X$ we refer to \cite[Ch. 1]{BS}).

Hence, (\ref{ineq}) will follow from
\begin{equation}\label{in1}
\|Mf\|_{X}\le c(n)\sup_{{\mathscr{D}},{\mathcal S}}\|{\mathcal A}_{{\mathscr{D}},{\mathcal S}}f\|_{X} \quad(f\ge 0)
\end{equation}
and
\begin{equation}\label{in2}
\sup_{{\mathcal S}\in {\mathcal D}}\|{\mathcal T}_{\mathcal S,m}f\|_{X} \le
c(n)m\sup_{{\mathscr{D}},{\mathcal S}}\|{\mathcal A}_{{\mathscr{D}},{\mathcal S}}f\|_{X}  \quad(f\ge 0).
\end{equation}

Inequality (\ref{in1}) was proved in \cite{L2}; we give the proof here for the sake of the completeness.
The proof is just a combination of Proposition \ref{prhp} and the Calder\'on-Zygmund decomposition. First, by Proposition \ref{prhp},
\begin{equation}\label{interm}
Mf(x)\le 6^n\sum_{\a=1}^{2^n}M^{{\mathscr{D}}_{\a}}f(x).
\end{equation}
Second, by the Calder\'on-Zygmund decomposition, if
$$\O_k=\{x:M^df(x)>2^{(n+1)k}\}=\cup_{j}{Q_j^k}\quad\text{and}\quad E_j^k=Q_j^k\setminus \O_{k+1},$$
then the family $\{Q_j^k\}$ is sparse and
$$M^df(x)\le 2^{n+1}\sum_{k,j}f_{Q_j^k}\chi_{E_j^k}(x)\le 2^{n+1}{\mathcal A}f(x).$$
From this and from (\ref{interm}),
\begin{equation}\label{pb}
Mf(x)\le 2\cdot 12^n\sum_{\a=1}^{2^n}{\mathcal A}_{{\mathscr{D}}_{\a},{\mathcal S}_{\a}}f(x),
\end{equation}
where ${\mathcal S}_{\a}\in {\mathscr{D}}_{\a}$ depends on $f$.
This implies (\ref{in1}) with $c(n)=2\cdot 24^n$.

We turn now to the proof of (\ref{in2}). Fix a family ${\mathcal S}=\{Q_j^k\}\in{\mathcal D}$.
Applying Proposition \ref{prhp} again, we can decompose the cubes $Q_j^k$ into $2^n$ disjoint families $F_{\a}$ such that for any $Q_j^k\in F_{\a}$ there exists
a cube $Q_{j,\a}^k\in {\mathscr{D}}_{\a}$ such that $2^mQ_j^k\subset Q_{j,\a}^k$ and
$\ell_{Q_{j,\a}^k}\le 6\ell_{2^mQ_j^k}$. Hence,
$${\mathcal T}_{\mathcal S,m}f(x)\le 6^n\sum_{\a=1}^{2^n}\sum_{j,k:Q_j^k\in F_{\a}}f_{Q_{j,\a}^k}\chi_{Q_j^k}(x).$$
Set
$${\mathcal A}_{m,\a}f(x)=\sum_{j,k}f_{Q_{j,\a}^k}\chi_{Q_j^k}(x).$$
We have that (\ref{in2}) will follow from
\begin{equation}\label{in3}
\|{\mathcal A}_{m,\a}f\|_{X}\le c(n)m\sup_{{\mathscr{D}},{\mathcal S}}\|{\mathcal A}_{{\mathscr{D}},{\mathcal S}}f\|_{X}  \quad(f\ge 0).
\end{equation}

Consider the formal adjoint to ${\mathcal A}_{m,\a}$:
$${\mathcal A}_{m,\a}^{\star}f=\sum_{j,k}\Big(\frac{1}{|Q_{j,\a}^k|}\int_{Q_j^k}f\Big)\chi_{Q_{j,\a}^k}(x).$$

\begin{prop}\label{a2b} For any $m\in{\mathbb N}$,
$$\|{\mathcal A}_{m,\a}^{\star}f\|_{L^2}=\|{\mathcal A}_{m,\a}f\|_{L^2}\le 8\|f\|_{L^2}.$$
\end{prop}

\begin{proof} Set $E_j^k=Q_j^k\setminus \Omega_{k+1}$.
Observe that the sets $E_j^k$ are pairwise disjoint and $|Q_j^k|\le 2|E_j^k|$.
From this,
\begin{eqnarray*}
\int_{{\mathbb R}^n}({\mathcal A}_{m,\a}f)gdx=\sum_{k,j}f_{Q_{j,\a}^k}g_{Q_j^k}|Q_j^k|&\le& 2\sum_{k,j}\int_{E_j^k}(M^{{\mathscr{D}}_{\a}}f)(M^dg)dx\\
&\le& 2\int_{{\mathbb R}^n}(M^{{\mathscr{D}}_{\a}}f)(M^dg)dx.
\end{eqnarray*}
From this, using H\"older's inequality, the $L^2$ boundedness of $M^d$ and duality, we get the $L^2$ bound for ${\mathcal A}_{m,\a}$.
\end{proof}

\begin{lemma}\label{weaktype} For any $m\in{\mathbb N}$,
$$\|{\mathcal A}_{m,\a}^{\star}f\|_{L^{1,\infty}}\le c(n)m\|f\|_{L^1}.$$
\end{lemma}

\begin{proof} Set $\Omega=\{x:Mf(x)>\a\}$ and let $\Omega=\cup_lQ_l$ be a Whitney decomposition such that $3Q_l\subset \Omega$,
where $Q_l\in {\mathcal D}$ (see, e.g., \cite[p. 348]{BS}).
Set also
$$b_l=(f-f_{Q_l})\chi_{Q_l},\quad b=\sum_lb_l$$
and $g=f-b$. We have
\begin{eqnarray}
|\{x:|{\mathcal A}_{m,\a}^{\star}f(x)|>\a\}|&\le& |\Omega|+|\{x:|{\mathcal A}_{m,\a}^{\star}g(x)|>\a/2\}|\nonumber\\
&+&|\{x\in \Omega^c:|{\mathcal A}_{m,\a}^{\star}b(x)|>\a/2\}|.\label{set}
\end{eqnarray}
Further, $|\Omega|\le \frac{c(n)}{\a}\|f\|_{L_1}$, and, by the $L^2$ boundedness of ${\mathcal A}_{m,\a}^{\star}$,
$$|\{x:|{\mathcal A}_{m,\a}^{\star}g(x)|>\a/2\}|\le \frac{4}{\a^2}\|{\mathcal A}_{m,\a}^{\star}g\|_{L^2}^2\le \frac{c}{\a^2}\|g\|_{L^2}^2\le \frac{c}{\a}\|g\|_{L^1}\le \frac{c}{\a}\|f\|_{L^1}$$
(we have used here that $g\le c\a$).

It remains therefore to estimate the term in (\ref{set}). For $x\in \Omega^c$ consider
$${\mathcal A}_{m,\a}^{\star}b(x)=\sum_l\sum_{k,j}\Big(\frac{1}{|Q_{j,\a}^k|}\int_{Q_j^k}b_l\Big)\chi_{Q_{j,\a}^k}(x).$$
The second sum is taken over those cubes $Q_j^k$ for which $Q_j^k\cap Q_l\not=\emptyset$. If $Q_l\subseteq Q_j^k$, then
$(b_l)_{Q_j^k}=0$. Therefore one can assume that $Q_j^k\subset Q_l$. On the other hand, $Q_{j,\a}^k\cap \Omega^c\not=\emptyset$.
Since $3Q_l\subset \Omega$, we have that $Q_l\subset 3Q_{j,\a}^{k}$. Hence
$$\ell_{Q_l}\le 3\ell_{Q_{j,\a}^k}\le 18\cdot 2^m\ell_{Q_j^k}.$$
The family of all dyadic cubes $Q$ for which $Q\subset Q_l$ and $\ell_{Q_l}\le 18\cdot 2^m\ell_{Q}$
can be decomposed into
$m+4$ families of disjoint cubes of equal length. Therefore,
$$\sum_{k,j:Q_j^k\subset Q_l\subset 3Q_{j,\a}^{k}}
\chi_{Q_j^k}\le (m+4)\chi_{Q_l}.$$
From this we get
\begin{eqnarray*}
&&|\{x\in \Omega^c:|{\mathcal A}_{m,\a}^{\star}b(x)|>\a/2\}|\le \frac{2}{\a}\|{\mathcal A}_{m,\a}^{\star}b\|_{L^1(\Omega^c)}\\
&&\le \frac{2}{\a}\sum_l\sum_{k,j:Q_j^k\subset Q_l\subset 3Q_{j,\a}^{k}}\int_{Q_j^k}|b_l|dx\le \frac{2(m+4)}{\a}\sum_l
\int_{Q_l}|b_l|dx\\
&&\le\frac{4(m+4)}{\a}\|f\|_{L^1}.
\end{eqnarray*}
The proof is complete.
\end{proof}

\begin{lemma}\label{locosc} For any cube $Q\in {\mathscr{D}}_{\a}$,
$$\o_{\la_n}({\mathcal A}_{m,\a}^{\star}f;Q)\le c(n)mf_Q\quad(\la_n=1/2^{n+2}).$$
\end{lemma}

\begin{proof} For $x\in Q$,
$$\sum_{k,j:Q\subseteq Q_{j,\a}^k}\Big(\frac{1}{|Q_{j,\a}^k|}\int_{Q_j^k}f\Big)\chi_{Q_{j,\a}^k}(x)=\sum_{k,j:Q\subseteq Q_{j,\a}^k}\Big(\frac{1}{|Q_{j,\a}^k|}\int_{Q_j^k}f\Big)\equiv c.$$
Hence
$$|{\mathcal A}_{m,\a}^{\star}f(x)-c|\chi_Q(x)=
\sum_{k,j:Q_{j,\a}^k\subset Q}\Big(\frac{1}{|Q_{j,\a}^k|}\int_{Q_j^k}f\Big)\chi_{Q_{j,\a}^k}(x)
\le {\mathcal A}_{m,\a}^{\star}(f\chi_Q)(x).$$
From this and from Lemma \ref{weaktype},
$$\inf_c(({\mathcal A}_{m,\a}^{\star}f-c)\chi_Q)^*(\la_n|Q|)\le ({\mathcal A}_{m,\a}^{\star}(f\chi_Q))^*(\la_n|Q|)\le c(n)mf_Q,$$
which completes the proof.
\end{proof}

We are ready now to prove (\ref{in3}). One can assume that the sum defining ${\mathcal A}_{m,\a}$ is finite.
Then $m_{{\mathcal A}_{m,\a}^{\star}f}(Q)=0$ for $Q$ big enough. Hence, By Lemma \ref{locosc} and Theorem \ref{decom1}, for a.e. $x\in~Q$
(where $Q\in {\mathscr{D}}_{\a}$),
$${\mathcal A}_{m,\a}^{\star}f(x)\le c(n)m\big(Mf(x)+{\mathcal A}_{{\mathcal S}_{\a},{\mathscr D}_{\a}}f(x)\big).$$
From this and from (\ref{pb}), for any $g\ge 0$ we have
\begin{eqnarray*}
&&\int_{{\mathbb R}^n}({\mathcal A}_{m,\a}f)gdx=\int_{{\mathbb R}^n}f({\mathcal A}_{m,\a}^{\star}g)dx\\
&&\le c_n m\sum_{\a=1}^{2^n+1}\int_{{\mathbb R}^n}f({\mathcal A}_{{\mathscr{D}}_{\a},{\mathcal S}_{\a}}g)dx\\
&&=c_n m\sum_{\a=1}^{2^n+1}\int_{{\mathbb R}^n}({\mathcal A}_{{\mathscr{D}}_{\a},{\mathcal S}_{\a}}f)gdx\le c'_n m\sup_{{\mathscr{D}},{\mathcal S}}
\|{\mathcal A}_{{\mathscr{D}},{\mathcal S}}f\|_X\|g\|_{X'}.
\end{eqnarray*}
Taking here the supremum over $g$ with $\|g\|_{X'}=1$ completes the proof.

\vskip 0.5cm
\noindent
{\bf Added in proof}.
We have just learned that T. Hyt\"onen, M. Lacey and C. P\'erez \cite{HLP} have also found
a proof of the $A_2$ conjecture avoiding a representation of $T$ in terms of Haar shifts.
The first step in this proof is the same: the ``local mean oscillation decomposition" combined
with Proposition \ref{oscsin} which reduces the problem to operators ${\mathcal A}_{m,\a}$.
In order to handle ${\mathcal A}_{m,\a}$, the authors use the result from \cite{HL} where it was
observed that this operator can be viewed as a positive Haar shift operator of complexity $m$.
As we have mentioned previously, our proof avoids completely the notion of the Haar shift operator, and
to bound ${\mathcal A}_{m,\a}$ we apply the decomposition again (as it is shown starting with Lemma \ref{weaktype}).


\begin{thebibliography}{99}
\bibitem{BS}
C. Bennett and R. Sharpley, {\it Interpolation of Operators},
Academic Press, New York, 1988.

\bibitem{Be}
O.V. Beznosova, {\it Linear bound for the dyadic paraproduct on
weighted Lebesgue space $L^2(w)$}, J. Funct. Anal., {\bf 255}
(2008), no. 4, 994-–1007.

\bibitem{B}
S.M. Buckley, {\it Estimates for operator norms on weighted spaces
and reverse Jensen inequalities}, Trans. Amer. Math. Soc., {\bf 340}
(1993), no. 1, 253--272.

\bibitem{CGCP}
J.M. Conde, J. Garc\'ia-Cuerva and J. Parcet, {\it
Sharp dyadic coverings and nondoubling Calder\'on-Zygmund theory},
preprint. Available at http://arxiv.org/abs/1201.3513

\bibitem{CMP1}
D. Cruz-Uribe, J.M. Martell and C. P\'erez, {\it Sharp weighted estimates for approximating dyadic operators},
Electron. Res. Announc. Math. Sci. {\bf 17} (2010), 12--19.

\bibitem{CMP2}
D. Cruz-Uribe, J.M. Martell and C. P\'erez, {\it Sharp weighted
estimates for classical operators}, Adv. Math., {\bf 229} (2012), no. 1, 408--441.


\bibitem{H}
T.P. Hyt\"onen, {\it The sharp weighted bound for general
Calder\'on-Zygmund operators}, to appear in Annals of Math. (2012). Available at
http://arxiv.org/abs/1007.4330


\bibitem{HL}
T.P. Hyt\"onen and M. Lacey, {\it The $A_p-A_{\infty}$ inequality for general Calder\'on--Zygmund operators},
to appear in Indiana Univ. Math. J. Available at http://arxiv.org/abs/1106.4797

\bibitem{HLMORSU}
T.P. Hyt\"onen, M.T. Lacey, H. Martikainen, T. Orponen, M.C.
Reguera, E.T. Sawyer and I. Uriarte-Tuero, {\it Weak and strong type
estimates for maximal truncations of Calder\'on-Zygmund operators on
$A_p$ weighted spaces}, preprint. Available at
http://arxiv.org/abs/1103.5229

\bibitem{HLP}
T.P. Hyt\"onen, M.T. Lacey and C. P\'erez,
{\it Non-probabilistic proof of the $A_2$ theorem, and sharp weighted bounds for the q-variation of singular integrals},
preprint. Available at http://arxiv.org/abs/1202.2229

\bibitem{HP}
T.P. Hyt\"onen and C. P\'erez, {\it Sharp weighted bounds involving
$A_{\infty}$}, to appear in J. Analysis\&PDE. Available at http://arxiv.org/abs/1103.5562

\bibitem{HPTV}
T.P. Hyt\"onen, C. P\'erez, S. Treil and A. Volberg, {\it Sharp
weighted estimates for dyadic shifts and the $A_2$ conjecture},
preprint. Available at http://arxiv.org/abs/1010.0755

\bibitem{JT}
B. Jawerth and  A. Torchinsky, {\it Local sharp maximal functions},
J. Approx. Theory, {\bf 43}\, (1985), 231--270.

\bibitem{L}
M.T. Lacey, {\it An $A_p$-$A_{\infty}$ inequality for the Hilbert
transform}, preprint. Available at http://arxiv.org/abs/1104.2199

\bibitem{La}
M.T. Lacey, {\it On the $A_2$ inequality for Calder\'on-Zygmund operators},
preprint. Available at http://arxiv.org/abs/1106.4802

\bibitem{LPR}
M.T. Lacey, S. Petermichl and M.C. Reguera, {\it Sharp $A_2$
inequality for Haar Shift Operators}, Math. Ann. {\bf 348} (2010),
no. 1, 127–-141.

\bibitem{L1}
A.K. Lerner, {\it A pointwise estimate for the local sharp maximal
function with applications to singular integrals}, Bull. London
Math. Soc., {\bf 42} (2010), no. 5, 843--856.

\bibitem{L2}
A.K. Lerner, {\it
On an estimate of Calder\'on-Zygmund operators by dyadic positive operators},
preprint. Available at http://arxiv.org/abs/1202.1860

\bibitem{NRV}
F. Nazarov, A. Reznikov and A. Volberg, 
{\it The proof of $A_2$ conjecture in a geometrically doubling metric space},
preprint. Available at http://arxiv.org/abs/1106.1342

\bibitem{NV}
F. Nazarov and A. Volberg, 
{\it A simple sharp weighted estimate of the dyadic shifts on metric spaces with geometric doubling},
preprint. Available at http://arxiv.org/abs/1104.4893


\bibitem{PTV}
C. P\'erez, S. Treil and A. Volberg, {\it On $A_2$ conjecture and
corona decomposition of weights}, preprint. Available at
http://arxiv.org/abs/1006.2630


\bibitem{P1}
S. Petermichl, {\it The sharp bound for the Hilbert transform on
weighted Lebesgue spaces in terms of the classical $A_p$-
characteristic}, Amer. J. Math., {\bf 129} (2007), no. 5,
1355--1375.

\bibitem{P2}
S. Petermichl, {\it The sharp weighted bound for the Riesz
transforms}, Proc. Amer. Math. Soc., {\bf 136} (2008), no. 4,
1237--1249.


\bibitem{PV}
S. Petermichl and A. Volberg, {\it Heating of the Ahlfors-Beurling
operator: weakly quasiregular maps on the plane are quasiregular},
Duke Math. J. {\bf 112} (2002), no. 2, 281--305.

\bibitem{T}
S. Treil, {\it Sharp $A_2$ estimates of Haar shifts via Bellman function},
preprint. Available at http://arxiv.org/abs/1105.2252

\end{thebibliography}
\end{document}